\documentclass[12pt]{amsart}
\usepackage{amsmath,amsfonts,amsthm,amscd,amssymb,mathrsfs,extarrows}
\usepackage{enumitem}
\usepackage{hyperref}
\usepackage{tikz, tikz-cd}
\usepackage{bm}
\newtheorem{theorem}{Theorem}[section]

\newtheorem{proposition}[theorem]{Proposition}
\newtheorem{lemma}[theorem]{Lemma}
\newtheorem{corollary}[theorem]{Corollary}
\theoremstyle{definition}

\newtheorem{remark}[theorem]{Remark}

\newcommand{\id}{{\rm{id}}}

\newcommand{\PP}{\mathbb{P}}
\newcommand{\CC}{\mathbb{C}}

\newcommand{\ZZ}{\mathbb{Z}}

\newcommand{\Pic}{{\rm{Pic}}}

\newcommand{\cO}{\mathcal O}
\newcommand{\cH}{\mathcal H}
\newcommand{\cA}{\mathcal A}

\newcommand{\cL}{\mathcal L}
\newcommand{\cV}{\mathcal V}

\newcommand{\et}{\chi_{\mathrm{top}}}

\newcommand{\sHom}{\mathcal{H}\!om}

\newcommand{\im}{\operatorname{im}}

\newcommand{\dP}{\mathrm{dP}}

\DeclareMathOperator{\Ker}{Ker}

\usepackage[letterpaper]{geometry}
\numberwithin{equation}{section}
\begin{document}
\title[Complex structures on $S^6$]
{A two-parameter family of complex structures on $S^6$}
\author{Jeff Viaclovsky}
\address{Department of Mathematics, University of California, Irvine, USA}
\email{jviaclov@uci.edu}
\date{September 26, 2026}

\begin{abstract}
We construct a two-parameter family of complex structures on the standard
six-sphere, starting from a rational elliptic surface with singular fibers
$III^*,I_1,I_1,I_1$. The resulting torus fibration over $\PP^1$ has one
multiple fiber of multiplicity four, one hexagonal fiber, and two rank-one
degenerations. These complex structures
are not biholomorphic to the original Alp\"oge--Claude examples. 
\end{abstract}
\maketitle
\tableofcontents

\section{Introduction}\label{int}

The purpose of this paper is to construct complex structures on $S^6$
which are not biholomorphic to the original examples of Alp\"oge--Claude
\cite{AC}. We follow Engel's construction from an elliptic surface,
a line bundle, and a lifted translation \cite{Engel}, but start with the
minimal regular model of
\begin{align}\label{we0}
 y^2=x^3+tx+1.
\end{align}
This rational elliptic surface $\pi:S\to\PP^1$ has a $III^*$ fiber at
infinity and three nodal fibers at the roots $p_1,p_2,p_3$ of $4t^3+27$.
The section $P(t)=(0,1)$ generates its Mordell--Weil group. For
$M=\cO_S(P-O)$, a suitable lift of translation by $-2P$ gives a family
of complex two-tori over the complement of these four values.

The threefold examples are obtained by taking a free $\ZZ$-quotient of
$M^\times|_{S_U}$, where
$U=\PP^1\setminus\{p_1,p_2,p_3,\infty\}$ and $S_U=\pi^{-1}(U)$.
This is a smooth torus fibration. We then complete this family to a
holomorphic map $f:X\to\PP^1$ with four exceptional fibers. Over $p_1$
and $p_2$, the special fibers have normalizations which are ruled over
elliptic curves; each singular fiber is obtained by identifying its zero
and infinity sections. These are rank-one degenerations. The special
fiber over $p_3$ has the same hexagonal local model as in
Alp\"oge--Claude: the normalization is a degree-$6$ del Pezzo surface,
and the fiber is obtained by identifying a cycle of $\PP^1$s in pairs;
the resulting singular fiber has three double curves and two triple
points. This degeneration is rank-two. The above three fibers are all
reduced.

At infinity, after the fourth-root base change $s=r^4$, where $s=t^{-1}$,
the family admits a completion by a smooth torus. Using a logarithmic
transform, we take the quotient by a free affine action of order four to
obtain a multiple fiber $4B_*$, where $B_*$ is a smooth bielliptic surface.
This fiber lies over the original point at infinity; no additional
exceptional value is introduced. All singular fibers have topological
Euler characteristic $0$, except for the hexagonal one, for which it is $2$.

Some properties of the fibrations $\pi$ and $f$ are summarized in
\eqref{fibers}. The multiplicity and topological Euler characteristic
are those of the surface fiber of $f$. The last column gives the order
of the linear monodromy on $H^1(F,\ZZ)$, where $F$ is a general surface
fiber.
\begin{align}\label{fibers}
\begin{array}{c|c|c|c|c|c}
 \text{base point}&\text{fiber of }\pi&
 \text{reduced fiber of }f &\text{multiplicity} & \chi_{\text{top}} & \text{monodromy order} \\\hline
 p_1&I_1&\text{rank-one}&1 & 0 & \infty\\
 p_2&I_1&\text{rank-one}&1 & 0 & \infty \\
 p_3&I_1&\text{hexagonal}&1 & 2 & \infty \\
 \infty&III^*&\text{bielliptic surface }B_*&4 & 0 & 4
\end{array}
\end{align}

We show that the resulting total space is diffeomorphic to $S^6$ and
that the complex structures vary in a family of dimension at least two,
and we identify the canonical bundle. Our main theorem is as follows. 
\begin{theorem}\label{main}
Fix a sufficiently small $\lambda_0\ne0$. There are positive constants
$\epsilon,\epsilon'$, with $\epsilon'<|\lambda_0|$, and a proper
holomorphic submersion
\begin{align}
 \mathfrak X\longrightarrow
 \{\mu\in\CC:|\mu|<\epsilon\}\times
 \{\lambda\in\CC:|\lambda-\lambda_0|<\epsilon'\}
\end{align}
whose fibers $X_{\mu,\lambda}$ satisfy the following.
\begin{enumerate}[label=\textup{(\roman*)}]
\item Each $X_{\mu,\lambda}$ is diffeomorphic to the standard $S^6$.
\item There is a holomorphic map $f:X_{\mu,\lambda}\to\PP^1$ with
connected fibers and general fiber a complex two-torus. Its exceptional
fibers have the types and multiplicities listed in \eqref{fibers}, with
$p_i$ replaced by the roots $p_i(\mu)$ of $4t^3+27(\mu t+1)^2$, labelled
holomorphically so that $p_i(0)=p_i$. The rank-one fibers have elliptic
ruled normalizations, and the hexagonal fiber is obtained by identifying
opposite boundary curves of $\dP_6$.
\item The canonical bundle is
\begin{align}\label{can0}
 K_{X_{\mu,\lambda}}\simeq f^*\cO_{\PP^1}(-1).
\end{align}
The map $f$ is the algebraic reduction and is defined by the complete
anticanonical pencil.
\item No $X_{\mu,\lambda}$ is biholomorphic to a member of the original
Alp\"oge--Claude family.
\item The image of this family in the local deformation space has
complex dimension two. In particular, the local deformation space at a
general member has dimension at least two.
\end{enumerate}
\end{theorem}

To carry out the construction, we use the standard quotient and local filling constructions from \cite{Engel}. We also use the integral specialization results of
\cite[Propositions~2.7--2.8]{Wang}. We do not need to repeat most of those arguments here.
The calculations which we do discuss in more detail for our configuration are the lifted translation, the order-four good-reduction comparison, and the integral monodromy and
integral homology calculations. 

Wang's different elliptic configuration, $IV^*+I_2+I_1+I_1$, gives a simply
connected threefold with $H_2=H_3=\ZZ/2$ \cite{Wang}. Other configurations
are considered by Yang \cite{Yang}; see also related work in \cite{Chen}.
Sections~\ref{input}--\ref{fill} construct our threefold, and
Sections~\ref{mon}--\ref{coh} determine its integral topology.
Section~\ref{geom} proves the canonical-bundle and deformation statements, and also gives a brief discussion of some of the Hodge numbers. 

\subsection{Statement}  The main inspiration for this article is Philip Engel's paper~\cite{Engel}, which was inspired by the Alp\"oge-Claude construction.
Based on \cite[Remark~8.5]{Engel} and after applying some other fiber and monodromy restrictions for an arbitrary fibration $f: S^6 \rightarrow \PP^1$ with normal crossing fibers which were known to the author, the author was led naturally to the above $(III^*,I_1, I_1, I_1)$ configuration. The construction below adapts Engel's method to this elliptic surface.
The author used ChatGPT 6 Astra to assist with drafting and revising
the exposition, developing parts of the arguments, and carrying out
most of the computations in Sections~\ref{mon} and \ref{coh}.
The author takes responsibility for the mathematical content of
the paper.

\subsection{Other related work}
As this paper was being finalized, the author received a preprint
from S\"onke Rollenske, joint with Wenfei Liu, giving a related
construction of complex structures on $S^6$ \cite{LR26}. Their construction
starts from a rational elliptic surface with singular fibers
$III^*,II,I_1$, whereas the construction here uses
$III^*,I_1,I_1,I_1$. The present work was developed independently
of their preprint.

\subsection{Acknowledgements} The author would like to thank Nobuhiro Honda for numerous discussions regarding the possibility of a complex structure on $S^6$ through our joint work~\cite{HV}.  The author would also like to thank S\"onke Rollenske for prior discussions regarding complex structures on $S^6$. The author was partially supported by NSF Grant DMS-2404195.

\section{The elliptic surface and the linearization}\label{input}

Let $\pi:S\to\PP^1$ be the minimal regular model of \eqref{we0}, and
let $O$ be its zero section. The finite zeros of $4t^3+27$ will be
labelled $p_1,p_2,p_3$. Paths to these points, and hence their ordering,
will be chosen in Section~\ref{mon}. Set
\begin{align}
 U=\PP^1\setminus\{\infty,p_1,p_2,p_3\},\qquad
 S_U=\pi^{-1}(U).
\end{align}

The notation $P=(0,1)$ denotes the section $t\mapsto(t,0,1)$ in the
affine Weierstrass coordinates. The identity component of a singular
fiber is the component met by $O$.

\begin{lemma}\label{mw}
The singular fibers of $S$ are $III^*$ at infinity and $I_1$ at
$p_1,p_2,p_3$. The Mordell--Weil group is infinite cyclic, generated by
$P=(0,1)$. Moreover,
\begin{align}\label{height}
 P\cdot O=0,\qquad\langle P,P\rangle=\frac12,\qquad
 2P=\left(\frac{t^2}{4},-1-\frac{t^3}{8}\right).
\end{align}
Translation by $2P$ preserves every irreducible fiber component.
\end{lemma}

\begin{proof}
The discriminant is $-16(4t^3+27)$ and has three simple finite zeros.
At infinity put $s=t^{-1}$, $X=s^2x$, and $Y=s^3y$. The equation becomes
\begin{align}\label{inf}
 Y^2=X^3+s^3X+s^6.
\end{align}
The minimal Weierstrass coefficients have orders three and six, and
the discriminant has order nine. This is a $III^*$ fiber. The
fundamental line bundle of the elliptic surface is $\cO_{\PP^1}(1)$,
so $S$ is rational and $K_S=\pi^*\cO_{\PP^1}(-1)$.

The reducible-fiber root lattice is $E_7$. The Shioda--Tate formula
gives Mordell--Weil rank one. The section $P$ is disjoint from $O$ and
meets the nonidentity multiplicity-one component at $III^*$; the
identity component is the strict transform containing the smooth
part of the Weierstrass cubic. The local height contribution is
$3/2$, hence $\langle P,P\rangle=2-3/2=1/2$.
For any nonzero section, the height is
$2+2Q\cdot O-\operatorname{contr}_\infty(Q)$, with contribution either
zero or $3/2$. It is therefore at least $1/2$. There is no nonzero
torsion section, and $P$ cannot be a nontrivial multiple. The height
formulas are those of \cite{SS,Shioda}.

To compute $2P$, use the tangent-line definition of the elliptic
curve group law. Differentiating $y^2=x^3+tx+1$ with $t$ fixed gives
$2y\,dy/dx=3x^2+t$. At $P=(0,1)$ the slope is $t/2$, so the tangent
line is $y=1+(t/2)x$. Substitution into the cubic gives
\begin{align*}
 \left(1+\frac t2x\right)^2=x^3+tx+1
 \quad\Longleftrightarrow\quad x^2\left(x-\frac{t^2}{4}\right)=0.
\end{align*}
Thus the third intersection point, counted with multiplicity, is
$Q=(t^2/4,1+t^3/8)$. The group law gives $P+P+Q=O$, and negation is
$(x,y)\mapsto(x,-y)$. Hence $2P=-Q$, which is the formula in
\eqref{height}. In the infinity chart it limits to
$(X,Y)=(1/4,-1/8)$, a smooth point of the identity component. The component group at $III^*$ has order
two; the other singular fibers are irreducible. Translation by $2P$
therefore preserves every component.
\end{proof}

Set $M=\cO_S(P-O)$, and write $M^\times$ for the complement of its
zero section. We use $t_{2P}$ for translation on $S$ and reserve $F$
for a smooth surface fiber when no confusion is possible.

\begin{lemma}\label{lin}
There is an isomorphism
\begin{align}\label{hom}
 \sHom(t_{2P}^*M,M)\simeq\pi^*\cO_{\PP^1}(1).
\end{align}
Consequently, up to multiplication by a nonzero constant, there is a
unique holomorphic bundle map $\Phi_0:t_{2P}^*M\to M$ whose zero divisor
is exactly $S_{p_3}$. It is linear on each complex-line fiber and is
invertible away from $S_{p_3}$.
\end{lemma}

\begin{proof}
Set
\begin{align}\label{rat}
 D=P-O-(-P)+(-2P),\qquad
 g=\frac{y-1-\frac t2x}{x}.
\end{align}
Here $D$ is the indicated formal difference of section curves on $S$.
The Hom bundle in \eqref{hom} is $\cO_S(D)$. On every smooth fiber
its degree-zero divisor has sum zero in the group law, so its
restriction is trivial. Moreover $D$ has intersection zero with
every fiber component, since translation by $2P$ preserves the
components. A fiberwise trivial line bundle with these intersection
numbers is pulled back from the base. One can see this directly by
subtracting the divisor of a rational trivialization on the generic
fiber: the remaining vertical divisor lies in the kernel of each
fiber's intersection matrix, and is therefore a sum of full fibers.
There are no multiple fibers on $S$. Intersecting with $O$ gives
$D\cdot O=1$, proving \eqref{hom}. A section of $\cO_{\PP^1}(1)$
with its zero at $p_3$ is unique up to scale, which proves the existence
and uniqueness of $\Phi_0$.

We will need an explicit expression for this map at infinity. The
numerator of $g$ is the tangent line at $P$. On a smooth fiber its
divisor is $2P+(-2P)-3O$, while $x$ has divisor $P+(-P)-2O$. Thus its
horizontal divisor is $D$. There are no vertical contributions at a
finite fiber. At infinity,
\begin{align}
 g=s^{-1}\frac{Y-X/2-s^3}{X},
\end{align}
which has order $-1$ on the multiplicity-one identity component.
The intersection-matrix argument just used shows that the complete
vertical contribution is $-S_\infty$. Hence
$\operatorname{div}(g)=D-S_\infty$.

To interpret this as a bundle map, let $\eta_D$ be the meromorphic
section of $\cO_S(D)$ whose divisor is $D$. In a nonvanishing local
holomorphic frame $e_D$, this means that $\eta_D=d\,e_D$, where the
meromorphic coefficient $d$ has exactly the zeros and poles prescribed
by $D$. This is the distinguished section associated with the divisor.
It is a section of a line bundle, not a meromorphic function with
divisor $D$ on $S$.

The section $\eta_D/g$ has divisor $S_\infty$. Multiplying it by
$t-p_3$, whose divisor is $S_{p_3}-S_\infty$, gives
\begin{align}\label{phi0}
 \Phi_0=\frac{t-p_3}{g}\eta_D,\qquad
 \operatorname{div}(\Phi_0)=S_{p_3}.
\end{align}
The zero at infinity of $\eta_D/g$ cancels the pole of $t-p_3$.
Thus the product is holomorphic everywhere, although the displayed
factors are meromorphic. Under \eqref{hom} it is simply a section of
$\pi^*\cO_{\PP^1}(1)$ with its zero over $p_3$; the formula fixes a
trivialization for the calculation in Lemma~\ref{comp}.
\end{proof}

Here a bundle map is a holomorphic map which, in local trivializations,
has form $(x,v)\mapsto(x,a(x)v)$, with $a$ holomorphic.
Since $(t_{2P}^*M)_x=M_{x+2P}$, its map on the fiber at $x$ is
$\Phi_{0,x}:M_{x+2P}\to M_x$. It is invertible off $S_{p_3}$ and is
the zero linear map over $S_{p_3}$.

A section of the Hom line bundle is therefore a fiberwise linear lift
of the translation. By \eqref{hom}, the line bundle parametrizing such lifts is
pulled back from a degree-one line bundle on the base. Prescribing its
single zero at $p_3$ leaves only a scalar parameter. For a sufficiently
small nonzero $\lambda$, put $\Phi=\lambda\Phi_0$.
The automorphism of $M^\times|_{S_U}$ used in the quotient is
\begin{align}\label{lift}
 \mathcal F(v_x)=\Phi_{x-2P}(v_x)\in M_{x-2P}^\times.
\end{align}
It covers translation by $-2P$. This sign is fixed throughout the
monodromy and topology calculations.

We use the following standard consequence of the contracting quotient
construction in \cite{Engel}; see also \cite[Proposition~2.4]{Wang}.

\begin{proposition}\label{open}
For sufficiently small $\lambda\ne0$, the quotient
\begin{align}\label{openx}
 X_U=M^\times|_{S_U}/\langle\mathcal F\rangle\longrightarrow U
\end{align}
is a proper holomorphic submersion with complex two-torus fibers.
\end{proposition}

The hypotheses hold here because $M_t$ has degree zero, and
\eqref{lift} is invertible over $U$ and uniformly contracting after
scaling $\lambda$. Since $\deg(M|_O)=1$, choose a section $e$ with
a simple zero at $p_3$.
Its image in $X_U$ supplies the origins used in the local comparisons.

\section{The local completions}\label{fill}

We need to fill the punctured family at precisely the four values
$p_1,p_2,p_3,\infty$. At the finite values we use the standard Mumford
fillings. At infinity we first pass to the smooth family supplied by good
reduction and then take a twisted finite quotient. Throughout, the
punctured-disc identifications are marked by the section $e$ above.

\subsection{The three nodal values}

Let $q$ be a local coordinate centered at one of $p_1,p_2,p_3$.
Multiplicative uniformization gives two identifications on $(\CC^*)^2$:
\begin{align}\label{periods}
 h(z,v)&=(q\alpha(q)z,\beta(q)v),\notag\\
 k(z,v)&=(\gamma(q)z,\lambda q^\varepsilon\delta(q)v),
\end{align}
where $\alpha,\beta,\gamma,\delta$ are holomorphic and nowhere zero, and
$\varepsilon=0$ at $p_1,p_2$, while $\varepsilon=1$ at $p_3$.
The first factor of $q$ comes from the nodal elliptic curve. The second,
when present, comes from the zero of $\Phi_0$. The other factors are units
because $P$ meets the smooth locus of each nodal fiber. This is the
multiplicative-period description used in \cite{Engel} and
\cite[Lemma~2.6]{Wang}, with translation by $-2P$ here.

We use the one-component rank-one and the full-lattice rank-two
Mumford fillings. We record their properties and explain how their
normalizations arise from \eqref{periods}.

\begin{proposition}\label{mum}
The resulting fillings have smooth total spaces and reduced irreducible
central fibers. At $p_1,p_2$, the normalization is
$\PP_E(\cO_E\oplus L)$, with $E$ elliptic and $L\in\Pic^0(E)$; its two
boundary sections are identified by a translation. At $p_3$, the
normalization is $\dP_6$, with opposite boundary curves identified.
The respective Euler characteristics are $0,0,2$.

The section $e$ extends through the smooth locus of each filling.
Let $W$ be a central fiber and $F$ a nearby torus. The monodromy
$T:H^1(F,\ZZ)\to H^1(F,\ZZ)$ satisfies $(T-I)^2=0$, and its image
$(T-I)H^1(F,\ZZ)$ is a primitive sublattice, of rank one at
$p_1,p_2$ and rank two at $p_3$. Specialization induces isomorphisms
\begin{align}\label{sp}
 H^q(W,\ZZ)\xrightarrow{\ \simeq\ }H^q(F,\ZZ)^T,
 \qquad 0\leq q\leq4,
\end{align}
where in degree $q$ we use $T$ for the induced monodromy on
$H^q(F,\ZZ)$. Specialization on fundamental groups is surjective, with kernel
generated by the primitive vanishing cycles.
\end{proposition}

\begin{proof}
The pairs of orders of vanishing in \eqref{periods} are $(1,0),(0,0)$
at $p_1,p_2$, and $(1,0),(0,1)$ at $p_3$. The primitive rank-one and
full-lattice rank-two Mumford constructions \cite{Mumford,EGS} therefore
apply as in \cite[Propositions~2.7--2.8]{Wang}. Both rank-one elliptic
periods here already have order one.

We describe the rank-one normalization explicitly. The nodal filling
of the degenerating $z$-coordinate is covered by an infinite chain of
rational curves. The generator $h$ carries one component to the next,
so the quotient has one irreducible component. Before identifying its
two boundary sections, the remaining generator $k$ gives the smooth
surface
\begin{align}\label{norm}
 \widetilde W=\bigl(\PP^1_z\times\CC^*_v\bigr)
 /\bigl\langle(z,v)\mapsto(\gamma(0)z,cv)\bigr\rangle,
 \qquad c=\lambda\delta(0).
\end{align}
For $|\lambda|$ small we have $0<|c|<1$. The projection to the second
factor descends to a $\PP^1$-bundle over the elliptic curve
\begin{align}\label{dc}
 E=\CC^*/\langle c\rangle.
\end{align}
On the affine chart $z\ne\infty$, its transition function in the
$z$-coordinate is the constant $\gamma(0)$. This defines a flat line
bundle $L\in\Pic^0(E)$, and adjoining $z=\infty$ gives
$\widetilde W=\PP_E(\cO_E\oplus L)$. The loci $z=0$ and $z=\infty$
are disjoint sections, each isomorphic to $E$. The generator $h$
identifies them by the translation induced by $v\mapsto\beta(0)v$
(up to interchanging the two sections). Separating these two branches
recovers the smooth surface \eqref{norm}, so it is the normalization
of $W$. Its Euler characteristic, and that of the identified elliptic
curve, are zero. Hence $\et(W)=0$.

At $p_3$, the two period vectors generate $\ZZ^2$. The usual periodic
triangulation of the unit squares therefore gives exactly Engel's
hexagonal filling, whose normalization is $\dP_6$ with opposite
boundary curves identified; see \cite{Engel}. Its three rational double
curves and two triple points give
$\et(W)=\et(\dP_6)-3\et(\PP^1)+2=2$.

At $p_1,p_2$ the section $e$ has unit coordinates. At $p_3$, applying
$k^{-1}$ divides its second coordinate by $\lambda q\delta(q)$,
again giving unit coordinates. Thus it extends through the smooth
locus in the chosen charts. Taking logarithms of \eqref{periods}
adds one basic integral period for each factor of $q$, giving the
primitive monodromy asserted in the statement. Finally,
\cite[Propositions~2.7--2.8]{Wang} give \eqref{sp} and the asserted
vanishing-cycle description of the fundamental group for precisely
these primitive models. We use their integral statements, not only
the rational invariant-cycle theorem.
\end{proof}

\subsection{Good reduction over infinity}

This subsection concerns the point $\infty$ which already occurs among
the four singular values of $\pi:S\to\PP^1$. 
Let $s=t^{-1}$ and recall \eqref{inf}. After the base change $s=r^4$,
set $X=r^6x'$, $Y=r^9y'$. The smooth elliptic model is
\begin{align}\label{good}
 y'^2=x'^3+x'+r^6,\qquad P'(r)=(0,r^3).
\end{align}
Its central elliptic curve is smooth, and $P'(0)=(0,0)$ is a nonzero
two-torsion point. In particular $2P'(0)=O'$. On this model set
$M'=\cO(P'-O')$.

The following is a direct analogue of
\cite[Proposition~2.9]{Wang}, using Engel's potential-good-reduction
construction \cite{Engel}. We include the proof because the coordinate
powers and the comparison of the chosen sections are different.

\begin{lemma}\label{comp}
The pulled-back punctured torus family extends to a smooth proper family
$\cA\to\Delta_r$. Its central fiber $A_*$ is an abelian surface.
The punctured comparison can be chosen to preserve $e$, and the deck
map extends to an order-four group action $\rho$ on $\cA$.
\end{lemma}

\begin{proof}
For the function $g$ in \eqref{rat}, put
\begin{align}\label{G}
 g'&=\frac{y'-r^3-\frac{x'}{2r^3}}{x'},&
 G&=r^3g'=\frac{r^3(y'-r^3)}{x'}-\frac12.
\end{align}
Substitution gives
\begin{align}\label{compf}
 g=r^{-1}g',\qquad
 \frac{t-p_3}{g}=\frac{1-p_3r^4}{G}.
\end{align}
The divisor of $G$ is $D'=P'-O'-(-P')+(-2P')$.
There is no vertical term, since $G=-1/2$ at the generic point of the
central curve. Let $\eta_{D'}$ denote the meromorphic section of
$\cO(D')$ with divisor $D'$, using the convention explained in the
proof of Lemma~\ref{lin}. The extended lift is therefore
\begin{align}\label{extlin}
 \Phi'=\lambda(1-p_3r^4)\frac{\eta_{D'}}{G}:
 t_{2P'}^*M'\longrightarrow M'.
\end{align}
The quotient $\eta_{D'}/G$ is holomorphic and nowhere zero, including
at points where the separate meromorphic factors have zeros or poles.
After shrinking the disc, \eqref{extlin} is invertible, and its
sufficiently small scalar multiple is contracting. The same quotient
result as in Proposition~\ref{open} gives $\cA\to\Delta_r$.

Near the zero sections $O$ and $O'$, use the local coordinates
$z=-X/Y$ and $z'=-x'/y'$ along the elliptic curves. Both vanish simply
at the corresponding zero section, and the change of variables gives
\begin{align}\label{con}
 z=-X/Y=r^{-3}z',\qquad z'=-x'/y'.
\end{align}
This also compares the line bundles along those sections. Indeed,
$P$ is disjoint from $O$ here, so $\cO(P)$ is trivial near $O$,
whereas $\cO(-O)$ is the ideal of $O$, locally generated by $z$.
Its restriction to $O$ is generated by the class of $z$ modulo $z^2$.
The same description holds for $M'$ using $z'$. Thus \eqref{con}
identifies these local frames with a factor $r^{-3}$.

The chosen section $e$ is nonzero near infinity. The natural punctured
comparison therefore sends it to $r^{-3}$ times a regular nonzero
section, up to a holomorphic unit. Multiplying the comparison by $r^3$
and the inverse unit makes it preserve $e$. This scalar is pulled back
from the base and commutes with the lifted translation; it does not
change \eqref{compf}.

The deck map is
\begin{align}\label{deck}
 (r,x',y')\longmapsto(-ir,-x',iy').
\end{align}
It preserves $P',O'$ and $G$. Normalize its line-bundle lift to preserve
the chosen vector over $O'$. Its fourth power is then the identity,
and it commutes with \eqref{extlin}. The induced action $\rho$ on
$\cA$ preserves the origin. Finally, $M'|_{r=0}$ has order two, and
the lifted translation covers the identity on the central elliptic
curve. After a finite elliptic isogeny its quotient is a product of
elliptic curves. Thus $A_*$ is an abelian surface.
\end{proof}

\subsection{The single logarithmic transformation}

The smooth central torus $A_*$ just constructed is upstairs over $r=0$.
It is not an additional fiber of the threefold over the original base.
The untwisted action $\rho$ fixes its origin, so the direct quotient is
not the smooth filling we want. We use the standard logarithmic
modification of \cite{Engel}: replace $\rho$ by its composition with
a torsion translation, and then descend to $s=r^4$.

Let $\Lambda$ be the homology lattice of $\cA$ and $H_*$ the deck action.
Section~\ref{mon} specifies an invariant integral vector $\ell$ and an
invariant integral covector $\psi$ with
\begin{align}\label{prim}
 H_*\ell=\ell,\qquad \psi(\ell)=1.
\end{align}
Use the holomorphic four-torsion section $a=\ell/4$ and set
\begin{align}\label{aff}
 \widetilde\rho=t_a\circ\rho,\qquad
 Y_\infty=\cA/\langle\widetilde\rho\rangle\longrightarrow\Delta_s.
\end{align}
The action has order four. It is free on $A_*$ because its $j$-th power
shifts the circle coordinate defined by $\psi$ by $j/4$, for
$j=1,2,3$, and it is free away from $r=0$ by its action on the base.
Consequently $Y_\infty$ is smooth and its central fiber is $4B_*$,
where $B_*=A_*/\langle\widetilde\rho\rangle$ is a smooth bielliptic
surface with canonical order four.

The standard punctured identification with $X_U$ is translation by
\begin{align}\label{loggl}
 b(r)=\pm\frac{\log r}{2\pi i}\ell(r)
 \quad\text{modulo the period lattice}.
\end{align}
This conjugates the linear and affine deck actions; see the
logarithmic-gluing construction in~\cite{Engel}. We fix the sign and
lift of the base meridian $\gamma_*$ by the relation
\begin{align}\label{mer}
 \gamma_*^4=\ell.
\end{align}
The operation itself is standard; the integral choice \eqref{mer}
is part of our construction and will be used in the topological calculations below.

Gluing $Y_\infty$ and the three nodal fillings to $X_U$ gives a compact
smooth complex threefold and a proper holomorphic map
\begin{align}\label{global}
 f:X=X_{0,\lambda}\longrightarrow\PP^1.
\end{align}
All four identifications exist for a common sufficiently small bound
on $|\lambda|$. The fiber at $\infty$ is the single multiple fiber
$4B_*$; the other three are the reduced fibers described above.

\section{The integral monodromy}\label{mon}

We use cohomological monodromy, with matrices acting on coefficient
columns. Thus on the ordinary dual homology lattice the matrix is
$(T^{-1})^t$. We choose positively oriented meridians in the order
$*,1,2,3$, with product equal to one; the symbol $*$ refers to
infinity.

\subsection{The elliptic factor and the logarithm of the section}
The elliptic local system can be marked so that its monodromies are
\begin{align}\label{ellm}
 A_*=\begin{pmatrix}0&-1\\1&0\end{pmatrix},\qquad
 A_1=A_3=\begin{pmatrix}1&1\\0&1\end{pmatrix},\qquad
 A_2=\begin{pmatrix}1&0\\-1&1\end{pmatrix}.
\end{align}
To justify this simultaneous choice, first choose a marking for which
the monodromy at the $III^*$ fiber is the displayed matrix $A_*$.
For an initial choice of paths to the three nodal fibers, denote their
monodromies by $G_1,G_2,G_3$. Each $G_i$ is conjugate to $A_1$, and
the relation among the base meridians gives
\begin{align}
 G_1G_2G_3=A_*^{-1}
 =\begin{pmatrix}0&1\\-1&0\end{pmatrix}.
\end{align}
This is the monodromy matrix of a type $III$ fiber. By
\cite[Theorem~19 and Table~1]{CV}, every factorization of this matrix
into three nodal monodromies is Hurwitz equivalent to
$(A_2,A_1,A_2)$. Hurwitz moves correspond to changing the paths to
the critical values. Simultaneous conjugation by $A_*$ interchanges
$A_1$ and $A_2$ and leaves $A_*$ unchanged. Thus, after choosing the
paths and labeling the finite critical values accordingly, we obtain
\eqref{ellm}. In particular,
\begin{align}\label{ellrel}
 A_*A_1A_2A_3=I.
\end{align}

We must also record the motion of the section $P$ relative to the
period lattice. On a small simply connected open set in $U$, choose
a nonzero holomorphic differential $\vartheta_t$ on the elliptic curve
$S_t$, for example $dx/(2y)$, and a continuously varying integral basis
$a_t,b_t$ of its first homology. Write
$\omega_1(t)=\int_{a_t}\vartheta_t$ and
$\omega_2(t)=\int_{b_t}\vartheta_t$. Under the uniformization
$S_t\simeq\CC/(\ZZ\omega_1(t)+\ZZ\omega_2(t))$, a path from $O(t)$
to $P(t)$ defines
\begin{align*}
 z_P(t)=\int_{O(t)}^{P(t)}\vartheta_t.
\end{align*}
The path can be chosen locally so that $z_P$ is holomorphic. Changing
the path adds $r_1\omega_1+r_2\omega_2$, with $r_1,r_2\in\ZZ$.
This holomorphic lift to $\CC$ is the \emph{elliptic logarithm} of
$P$. It is not the ordinary logarithm of either coordinate of
$P=(0,1)$.

After continuation around a puncture, the point $P(t)$ returns to
itself, but its chosen lift can change by an integral period. The
vector $k_i\in\ZZ^2$ records the two coefficients of that period.
We use the intersection pairing to identify the elliptic period
lattice with the lattice carrying the matrices $A_i$ in \eqref{ellm}.
With our convention for left actions, the resulting affine data are
encoded by $\left(\begin{smallmatrix}A_i&k_i\\0&1\end{smallmatrix}\right)$.
Changing the initial lift by an integral vector $r$ replaces $k_i$ by
$k_i+(I-A_i)r$. Thus the vectors are not individually canonical;
their classes modulo these changes are the winding data of $P$.

At a nodal value, write the nearby elliptic curve multiplicatively
as $\CC^*/\langle Q(q)\rangle$, where $Q$ has a simple zero. Since
$P$ meets the smooth locus of the nodal fiber, it is represented by
a holomorphic function $u_P(q)$ which is nonzero at $q=0$. After
shrinking the disc, $u_P$ has a single-valued ordinary logarithm.
In this local choice the elliptic logarithm has no jump. Hence its
winding vector in any marking lies in $\im(A_i-I)$. The matrices in
\eqref{ellm} therefore give
\begin{align}\label{kabc}
 k_1=(a,0),\qquad k_2=(0,b),\qquad k_3=(c,0).
\end{align}
The product of the four affine matrices is the identity. Its last
column gives
\begin{align*}
 k_*+A_*k_1+A_*A_1k_2+A_*A_1A_2k_3=0,
\end{align*}
and substitution yields
\begin{align}
 k_*=(b-c,-a-b).
\end{align}
For $r=(r_1,r_2)$, the change $k_i\mapsto k_i+(I-A_i)r$ replaces
$(a,b,c)$ by $(a-r_2,b+r_1,c-r_2)$. Hence the only integer left after
this change is $d=a-c$.

\begin{lemma}\label{k}
The logarithms of $P$ can be chosen so that, for some $d\in\ZZ$,
\begin{align}\label{kc}
 k_*=(0,-d),\qquad k_1=(d,0),\qquad k_2=k_3=0.
\end{align}
\end{lemma}

\begin{proof}
In the preceding change of lift, take $r_2=c$ and $r_1=-b$. Then
$(a,b,c)$ becomes $(a-c,0,0)$, which gives \eqref{kc}.
\end{proof}

The primitive monodromy at the hexagonal fiber will determine $d$.

\subsection{The rank-four extension}

Let $V=H^1(F,\ZZ)$. The semiabelian presentation gives a filtration
with successive ranks $1,2,1$: the first term is dual to the quotient
by the additional discrete period, the middle quotient is the
elliptic period system, and the last term comes from the $\CC^*$
direction. Choose an integral basis
\begin{align}\label{bas}
 V=\ZZ\langle\psi,u,w,\delta\rangle
\end{align}
compatible with this filtration. The coefficient two in translation
by $-2P$ gives the invariant alternating class
\begin{align}\label{xi}
 \xi=u\wedge w+2\psi\wedge\delta.
\end{align}
For clarity, the same coefficient can be read from local logarithmic
periods: in a basis $1,\tau$ of the elliptic lattice and a logarithm
$z$ of $P$, the two nonconstant multiplicative periods have logarithms
$(\tau,z)$ and $(-2z,\rho)$. Changing the elliptic logarithm by an
integral period changes their cross terms in the ratio $1:-2$.
This yields \eqref{xi} and the integral extension below. 

For elliptic monodromy $A$, elliptic-logarithm vector $k$, and an integer
$c$, the resulting cohomological matrix has the form
\begin{align}\label{jac}
 \mathcal T(A,k,c)=
 \begin{pmatrix}
 1&-(A^tJ_2k)^t&c\\
 0&A&k\\
 0&0&1
 \end{pmatrix},\qquad
 J_2=\begin{pmatrix}0&2\\-2&0\end{pmatrix}.
\end{align}
Indeed, the lower right extension records the changes of logarithm
of $P$. Invariance of \eqref{xi} determines the upper middle row,
while the scalar multiplier of the lifted translation gives the
remaining central integer $c$. This is the Jacobi extension of the
elliptic local system; compare \cite{Engel} and
\cite[Section~4.3]{Wang}.

\begin{proposition}\label{mat}
For the local completions and the section-preserving comparison fixed
in Section~\ref{fill}, the monodromy matrices in the basis
\eqref{bas} are
\begin{align}\label{ms1}
 T_*&=\begin{pmatrix}
 1&0&-2&-1\\0&0&-1&0\\0&1&0&-1\\0&0&0&1
 \end{pmatrix},&
 T_1&=\begin{pmatrix}
 1&0&2&2\\0&1&1&1\\0&0&1&0\\0&0&0&1
 \end{pmatrix},\\
 \label{ms2}
 T_2&=\begin{pmatrix}
 1&0&0&0\\0&1&0&0\\0&-1&1&0\\0&0&0&1
 \end{pmatrix},&
 T_3&=\begin{pmatrix}
 1&0&0&-1\\0&1&1&0\\0&0&1&0\\0&0&0&1
 \end{pmatrix}.
\end{align}
They satisfy
\begin{align}\label{rels}
 T_*T_1T_2T_3=I,\qquad T_*^4=I,\qquad
 (T_i-I)^2=0\quad(1\leq i\leq3),
\end{align}
and the three nilpotent ranks are $1,1,2$.
\end{proposition}

\begin{proof}
Keep the integer $d$ in \eqref{kc} undetermined and insert the elliptic
matrices into \eqref{jac}. At infinity this gives
\begin{align}
 T_*=
 \begin{pmatrix}
 1&0&-2d&c_*\\0&0&-1&0\\0&1&0&-d\\0&0&0&1
 \end{pmatrix}.
\end{align}
The central entry of $T_*^4-I$ is $4(c_*+d^2)$. By
Lemma~\ref{comp}, $T_*^4=I$, so $c_*=-d^2$.
At $p_1$, the two potentially nonzero columns of $T_1-I$ are
$(2d,1,0,0)^t$ and $(c_1,d,0,0)^t$. Rank one, from
Proposition~\ref{mum}, requires $c_1=2d^2$. The same calculation at
$p_2$ gives $c_2=0$. Multiplying the four matrices now leaves just the
central scalar $c_*+c_1+c_2+c_3$. The sphere relation makes it zero,
so $c_3=-d^2$.

In particular, at the hexagonal fiber we have
\begin{align}\label{dprim}
 T_3=\begin{pmatrix}
 1&0&0&-d^2\\0&1&1&0\\0&0&1&0\\0&0&0&1
 \end{pmatrix},\qquad
 (T_3-I)V=\ZZ u+\ZZ d^2\psi.
\end{align}
The primitive period vectors in Proposition~\ref{mum} give rank
two and primitive image. Rank two first excludes $d=0$. The subgroup
in \eqref{dprim} has index $d^2$ in its saturation
$\ZZ u+\ZZ\psi$, so primitivity forces $d^2=1$.
Changing the sign of the chosen marking gives $d=1$.

Substituting $d=1$ and the four central entries gives exactly
\eqref{ms1}--\eqref{ms2}. Direct multiplication proves \eqref{rels},
and the ranks follow by inspection. Notice that the primitive local
monodromy used to determine $d$ was obtained directly from the two
coefficient-one powers of $q$ in \eqref{periods}; it did not use the
global matrices being determined here. The logarithmic transformation
changes their affine translation data, but not their linear action on
cohomology. That translation was fixed by \eqref{mer}.
\end{proof}

Let $\widehat\psi,\widehat u,\widehat w,\widehat\delta$ denote the
ordinary dual basis of $\Lambda=H_1(F,\ZZ)$. The homological
monodromy at infinity is
\begin{align}\label{hs}
 H_*=(T_*^{-1})^t=
 \begin{pmatrix}
 1&0&0&0\\-2&0&-1&0\\0&1&0&0\\1&1&0&1
 \end{pmatrix}.
\end{align}
The vector used for the logarithmic transformation is
\begin{align}\label{ell}
 \ell=\widehat\psi-\widehat u-\widehat w.
\end{align}
Equations~\eqref{hs} and~\eqref{ell} give
$H_*\ell=\ell$ and $\psi(\ell)=1$, as required in \eqref{prim}.
This completes the monodromy computations. 

\section{Identification with \texorpdfstring{$S^6$}{S6}}\label{coh}

\subsection{The fundamental group}\label{fund}

\begin{proposition}\label{pi}
The threefold $X$ is simply connected.
\end{proposition}

\begin{proof}
Over $U$, the distinguished section splits the fundamental-group
sequence of the torus bundle. The resulting presentation has the
abelian fiber lattice $\Lambda$ and the four meridians
$\gamma_*,\gamma_1,\gamma_2,\gamma_3$, with their sphere relation
and the conjugation actions $H_i=(T_i^{-1})^t$.

The section extends over the three Mumford fillings, so their
meridians are trivial in the completed space. The primitive
vanishing-cycle subgroups, computed from the three matrices, are
\begin{align}\label{van}
 \im(H_1-I)&=\ZZ(\widehat w+\widehat\delta),\notag\\
 \im(H_2-I)&=\ZZ\widehat u,\notag\\
 \im(H_3-I)&=\ZZ\widehat w+\ZZ\widehat\delta.
\end{align}
Their sum is the primitive subgroup
$\ZZ\widehat u+\ZZ\widehat w+\ZZ\widehat\delta$.
The sphere relation also kills $\gamma_*$. The order-four filling
imposes \eqref{mer}, so $\ell=0$. Modulo \eqref{van}, this is the
relation $\widehat\psi=0$.

The boundary of each filling surjects on its fundamental group. For
the Mumford models this follows from the collapse description in
Proposition~\ref{mum}; it also follows by moving loops off the
central divisor. The same argument applies to the smooth multiple
fiber. Thus the fillings introduce no additional generators. After
the meridians are killed, the remaining group is a quotient of the
abelian fiber lattice, and the relations above kill that entire
lattice. 
\end{proof}

\subsection{Global invariants, coinvariants, and cohomology on the base}

We calculate the integral Leray spectral sequence of $f$. The local
specialization lattices, rather than just their rational spans, will
be needed. Write
\begin{align}\label{forms}
 \xi=u\wedge w+2\psi\wedge\delta,\qquad
 \phi=\psi\wedge u\wedge w,\qquad
 \nu=\psi\wedge u\wedge w\wedge\delta.
\end{align}
We choose $\nu$ as an integral top-degree generator; reversing it
changes none of the index calculations.

Let $j:U\hookrightarrow\PP^1$ be the inclusion, and let $\cV$ be the
rank-four local system on $U$ with monodromies \eqref{ms1}--\eqref{ms2}.
Put $\cL_q=\bigwedge^q\cV$. The brackets below denote classes in global
coinvariants. The stalk of $j_*\cL_q$ at a puncture is the sublattice
fixed by its local monodromy.

\begin{lemma}\label{inv}
The invariant and coinvariant lattices are free of rank one and have
primitive generators as follows:
\begin{align}\label{ivtab}
\begin{array}{c|c|c}
 q& (\bigwedge^qV)^{\mathrm{mon}}&
          (\bigwedge^qV)_{\mathrm{mon}}\\ \hline
 0&1&[1]\\
 1&\psi&[\delta]\\
 2&\xi&[\psi\wedge\delta]\\
 3&\phi&[u\wedge w\wedge\delta]\\
 4&\nu&[\nu].
\end{array}
\end{align}
In degree two,
\begin{align}\label{cv}
 [u\wedge w]=2[\psi\wedge\delta],\qquad
 [\xi]=4[\psi\wedge\delta].
\end{align}
Moreover,
\begin{align}\label{par0}
 H^1(\PP^1,j_*\cL_q)=0\qquad(0\leq q\leq4),
\end{align}
and
\begin{align}\label{h2base}
 H^2(\PP^1,j_*\cL_q)
 \simeq (\bigwedge^qV)_{\mathrm{mon}}.
\end{align}
\end{lemma}

\begin{proof}
All the assertions can be checked by integer elimination. We give
primitive-minor certificates to keep track of possible torsion.
In degree two use the ordered basis
\begin{align}\label{b2}
 (b_1,b_2,b_3,b_4,b_5,b_6)
 =(\psi u,\psi w,\psi\delta,uw,u\delta,w\delta),
\end{align}
where juxtaposition denotes exterior product. The images of $T_2-I$
contain $-b_2,-b_6$, those of $T_3-I$ contain $b_1,b_1+b_2+b_5$,
and $(T_1-I)b_6=2b_3+b_5-2b_2-b_4$.
Thus the relation lattice contains
\begin{align}
 b_1,\ b_2,\ b_5,\ b_6,\ b_4-2b_3.
\end{align}
Every column of every $\bigwedge^2T_i-I$ lies in their span. This
proves the degree-two coinvariant assertion integrally.
The equations $T_2x=x$, $T_3x=x$, $T_1x=x$ successively give
$x_1=x_5=0$, $x_2=x_6=0$, and $x_3=2x_4$. Their primitive solution
is $\xi$. In degrees one and three the same elimination kills all
but $\delta$ and $uw\delta$, respectively, in the coinvariants,
and leaves $\psi$ and $\phi$ as the invariants. The relations have
unit coefficients. Degrees zero and four are constant.

To compute $H^1(\PP^1,j_*\cL_q)$, set $L=\bigwedge^qV$,
$G_i=\bigwedge^qT_i$, and $R_i=G_i-I$. A cohomology class on the
punctured sphere is represented by a cocycle, specified by its values
$v_i\in L$ on the four meridians. The sphere relation gives
$v_*+G_*v_1+G_*G_1v_2+G_*G_1G_2v_3=0$.
Its restriction to a punctured disc is the class of $v_i$ in
$L/R_iL$. The low-degree exact sequence for $j$ identifies
$H^1(\PP^1,j_*\cL_q)$ with the classes whose restrictions to all
four punctured discs vanish. Thus each $v_i$ must lie in $R_iL$.
Coboundaries are the tuples $(R_ix)_i$, with $x\in L$.
Consequently the required group is $\Ker d_1/\im d_0$ for the complex
\begin{align}\label{pcx}
 L\xrightarrow{d_0}\bigoplus_{i=*,1,2,3}R_iL
 \xrightarrow{d_1}L,
\end{align}
where
\begin{align}
 d_0(x)&=(R_*x,R_1x,R_2x,R_3x),\notag\\
 d_1(v_*,v_1,v_2,v_3)
 &=v_*+G_*v_1+G_*G_1v_2+G_*G_1G_2v_3.
\end{align}
The identity $G_*G_1G_2G_3=I$ also verifies directly that
$d_1d_0=0$. This is the complex used in the proof of
\cite[Lemma~6.3]{Wang}. The ranks of the four image lattices are
\begin{align}\label{ranks}
 (2,1,1,2),\qquad(4,2,2,2),\qquad(2,1,1,2)
\end{align}
for $q=1,2,3$. The cokernel of $d_1$ is the global coinvariant
lattice. This follows by successively removing the prefixes
$G_*,G_*G_1,G_*G_1G_2$; the relation sublattice is unchanged.
Its rank is one by \eqref{ivtab}. Thus $\Ker d_1$ has rank
$3,5,3$, respectively, equal to the rank of $\im d_0$.

It remains to prove that $\im d_0$ is primitive. Regard $d_0$ first
as a map into $L^4$. In degrees one and three, take the ordered
bases $(\psi,u,w,\delta)$ and $(\phi,\psi u\delta,\psi w\delta,
uw\delta)$, and omit the invariant first vector. The rows consisting
of the third coordinate of $R_2x$, the second coordinate of $R_3x$,
and the first coordinate of $R_3x$ give the minor
\begin{align}\label{minor1}
 \begin{pmatrix}-1&0&0\\0&1&0\\0&0&-1\end{pmatrix}.
\end{align}
In degree two, omit $b_4$ from \eqref{b2}; it is legitimate because
$\xi=b_4+2b_3$ is primitive. On the remaining ordered columns
$(b_1,b_2,b_3,b_5,b_6)$, take the rows
$(R_2x)_2$, $(R_2x)_6$, $(R_3x)_5$, $(R_3x)_1$, and $(R_1x)_1$.
The minor is
\begin{align}\label{minor2}
 \begin{pmatrix}
 -1&0&0&0&0\\
 0&0&0&-1&0\\
 0&0&0&0&1\\
 0&1&0&1&1\\
 0&1&1&-2&0
 \end{pmatrix}.
\end{align}
Both determinants are one. Hence $\im d_0$ is primitive in $L^4$
and consequently in the middle group of \eqref{pcx}. Its quotient
inside $\Ker d_1$ is finite by the rank calculation and torsion-free
by primitivity, so it is zero. This proves \eqref{par0}; degrees
zero and four reduce to $H^1(\PP^1,\ZZ)=0$.

Finally, the exact sequence from $j_!\cL_q$ to $j_*\cL_q$ has a
skyscraper quotient. Its degree-two cohomology identifies
$H^2(\PP^1,j_*\cL_q)$ with $H_c^2(U,\cL_q)$. Integral
local-coefficient duality on the oriented punctured sphere identifies
the latter with the coinvariants, proving \eqref{h2base}.
\end{proof}

\subsection{Specialization at the multiple fiber}

Recall that $B_*$ is the smooth bielliptic surface underlying the
multiple fiber $4B_*$. Let $A_*$ be the central torus after the
fourth-root base change and let $p:A_*\to B_*$ be the free affine
quotient of degree four. Ordinary topological cohomology does not see
the multiplicity, so the central-fiber group is $H^q(B_*,\ZZ)$.
After identifying the cohomology of $A_*$ with that of a nearby torus,
the specialization map is
\begin{align*}
 p^*:H^q(B_*,\ZZ)\longrightarrow
 H^q(A_*,\ZZ)^{\langle\widetilde\rho\rangle}
 \simeq(\bigwedge^qV)^{T_*}.
\end{align*}
Thus an index here is taken inside the invariant sublattice, not inside
all of $H^q(A_*,\ZZ)$. Set
\begin{align}\label{hh}
 h=u-w+2\delta,\qquad
 \zeta=-\psi u\delta+\psi w\delta+uw\delta.
\end{align}

\begin{lemma}\label{indices}
The groups $H^q(B_*,\ZZ)$ are torsion-free, and the specialization
maps $p^*$ are injective. For $q=1,2,3,4$, let
\begin{align*}
 I_q=\bigl[(\bigwedge^qV)^{T_*}:p^*H^q(B_*,\ZZ)\bigr].
\end{align*}
Then $(I_1,I_2,I_3,I_4)=(4,2,2,4)$. Equivalently, the quotient of the
invariant lattice by the image of specialization is respectively
$\ZZ/4$, $\ZZ/2$, $\ZZ/2$, and $\ZZ/4$. These quotients are generated
by the cosets of $\psi,\xi,\phi,\nu$, respectively.
\end{lemma}

\begin{proof}
The topological properties of bielliptic surfaces are well known. We
briefly review the calculation for this particular affine quotient,
since we need its integral cohomology and the precise pullback
lattices. The fundamental group of $B_*$ has presentation
\begin{align}\label{bp}
 \left\langle\Lambda,g\ \big|\ [\Lambda,\Lambda]=1,\quad
 gxg^{-1}=H_*x,\quad g^4=\ell\right\rangle.
\end{align}
In the coinvariants of $\Lambda$ under $H_*$, the columns of
$H_*-I$ give
\begin{align}
 \widehat w=-\widehat u,\qquad
 \widehat\delta=2\widehat u.
\end{align}
The third relation is dependent. Thus these coinvariants are freely
generated by $\widehat\psi,\widehat u$, and $\ell$ maps to
$\widehat\psi$. Abelianizing \eqref{bp} imposes only
$4g=\widehat\psi$. It follows that
\begin{align}\label{bbet}
 H_1(B_*,\ZZ)=\ZZ^2.
\end{align}
Since $B_*$ is a closed oriented four-manifold, Poincar\'e duality and
the universal coefficient theorem imply that all its cohomology is
torsion-free. Its Euler characteristic is zero, so $b_2(B_*)=2$.
Its holomorphically torsion canonical bundle has zero integral first
Chern class, since $H^2$ has no torsion. Consequently $B_*$ is spin
and its intersection form is even and unimodular. The transfer for
the degree-four covering satisfies $p_!p^*=4\,\id$. Torsion-freeness
therefore makes $p^*$ injective in every degree.

The invariant lattice in degree one is
$\ZZ\psi\oplus\ZZ h$. An invariant integral covector extends over
\eqref{bp} precisely when its value on $\ell$ is divisible by four.
As $\psi(\ell)=1$ and $h(\ell)=0$, this gives
\begin{align}\label{sp1}
 p^*H^1(B_*,\ZZ)=4\ZZ\psi\oplus\ZZ h.
\end{align}
Thus the degree-one cokernel is $\ZZ/4$, generated by $\psi$.

In degrees two and three, integer elimination in $T_*-I$ gives the
primitive invariant bases
\begin{align}
 (\bigwedge^2V)^{T_*}=\ZZ\langle\psi h,\xi\rangle, \quad
 (\bigwedge^3V)^{T_*}=\ZZ\langle\phi,\zeta\rangle.
\end{align}
Their relevant pairings, evaluated on $\nu$, are
\begin{align}\label{pair}
 \bigl(\alpha_i\alpha_j\bigr)_{\alpha=(\psi h,\xi)}
 =\begin{pmatrix}0&2\\2&4\end{pmatrix}, \quad
 \bigl(\beta_i\gamma_j\bigr)_{\beta=(\psi,h),\ \gamma=(\phi,\zeta)}
 =\begin{pmatrix}0&1\\-2&0\end{pmatrix}.
\end{align}
Pullback multiplies intersection pairings by four. If $I_q$ denotes
the pullback index in degree $q$, unimodularity downstairs gives
\begin{align}
 I_2^2\cdot4=4^2,\qquad I_1I_3\cdot2=4^2.
\end{align}
Since $I_1=4$, this yields $I_2=I_3=2$. The class $\xi$ cannot
descend: $\xi^2=4\nu$ would give a class of odd square downstairs,
contradicting evenness. Likewise $\phi$ cannot descend, since $h$
does descend and $h\phi=-2\nu$ would give a nonintegral pairing
after division by four. They therefore generate the respective
order-two cokernels. In top degree, a covering of degree four has
pullback index four and cokernel generated by $\nu$.
\end{proof}

\subsection{The Leray page}

Write $\cH^q=R^qf_*\ZZ$. Proposition~\ref{mum} and
Lemma~\ref{indices} give exact sequences of sheaves
\begin{align}\label{sheaf}
 0\longrightarrow\cH^q\longrightarrow j_*\cL_q
 \longrightarrow Q_q\longrightarrow0,
\end{align}
where the only nonzero stalk of $Q_q$ is at infinity, and these
stalks are $\ZZ/4,\ZZ/2,\ZZ/2,\ZZ/4$ in degrees one through four.
The global invariant generators in Lemma~\ref{inv} map onto these
cokernels by Lemma~\ref{indices}. Taking cohomology in \eqref{sheaf}
therefore gives the integral $E_2$ page, where $\omega$ denotes the base
orientation class $\omega\in H^2(\PP^1,\ZZ)$:
\begin{align}\label{page}
\begin{array}{c|c|c|c}
 q&E_2^{0,q}&E_2^{1,q}&E_2^{2,q}\\ \hline
 4&4\ZZ\nu&0&\ZZ\omega[\nu]\\
 3&2\ZZ\phi&0&\ZZ\omega[uw\delta]\\
 2&2\ZZ\xi&0&\ZZ\omega[\psi\delta]\\
 1&4\ZZ\psi&0&\ZZ\omega[\delta]\\
 0&\ZZ&0&\ZZ\omega.
\end{array}
\end{align}

\begin{proposition}\label{d2}
With a common choice of orientation, the four nonzero differentials
from column zero in \eqref{page} are
\begin{align}\label{unit}
 d_2(4\psi)=\omega,\quad d_2(2\xi)=\omega[\delta],\quad
 d_2(2\phi)=\omega[\psi\delta],\quad d_2(4\nu)=\omega[uw\delta].
\end{align}
In particular, each is an isomorphism of infinite cyclic groups.
\end{proposition}

\begin{proof}
Put $a_1=4\psi$, $a_2=2\xi$, $a_3=2\phi$, $a_4=4\nu$, and write
\begin{align}
 d_2a_1=p\omega,\quad
 d_2a_2=q\omega[\delta],\quad
 d_2a_3=r\omega[\psi\delta],\quad
 d_2a_4=s\omega[uw\delta].
\end{align}
The primitive logarithmic shift gives
$p=(4\psi)(\ell)/4=1$, after fixing the orientation of the base
meridian. This also follows without a clutching formula:
Proposition~\ref{pi} implies $H^1(X,\ZZ)=0$, so $p\ne0$;
the term $E_\infty^{2,0}=\ZZ/p$ injects into $H^2(X,\ZZ)$, which is
torsion-free since $H_1(X,\ZZ)=0$. Thus $p=\pm1$.

The higher coefficients follow from multiplicativity, as in
\cite{Engel}. The identities in the exterior algebra are
\begin{align}\label{prod}
 a_1a_2=4a_3,\qquad a_2a_3=0,\qquad a_2^2=4a_4.
\end{align}
The derivation rule for $d_2$, together with \eqref{cv}, applied to
the first equality gives
\begin{align}
 4r=8p-4q,\qquad\text{hence}\qquad r=2p-q.
\end{align}
For the second equality use
\begin{align}
 \delta\phi=-\nu,\qquad \xi\psi\delta=\nu.
\end{align}
It gives $-2q+2r=0$ in the free group
$E_2^{2,4}=\ZZ\omega[\nu]$, so $r=q$.
The last equality in \eqref{prod} gives $4s=4q$ because
$[\xi\delta]=[uw\delta]$. Consequently $p=q=r=s=1$.
No separate higher-degree clutching calculation is needed.
\end{proof}

\begin{theorem}\label{sphere}
The compact complex threefold $X$ is diffeomorphic to the standard
six-sphere.
\end{theorem}

\begin{proof}
There are no columns beyond column two in \eqref{page}. By
Proposition~\ref{d2}, all terms except $E_\infty^{0,0}=\ZZ$ and
$E_\infty^{2,4}=\ZZ$ vanish. There are no further differentials or
extension problems. Hence
\begin{align}\label{homol}
 H^k(X,\ZZ)=
 \begin{cases}
 \ZZ,&k=0,6,\\
 0,&1\leq k\leq5.
 \end{cases}
\end{align}
Poincar\'e duality gives the same integral homology groups. Together
with Proposition~\ref{pi}, this makes $X$ a homotopy six-sphere:
Hurewicz and Whitehead applied to a degree-one map give a homotopy
equivalence with $S^6$. The group of smooth homotopy six-spheres is
zero by Kervaire--Milnor \cite{KM}. Thus $X$ is diffeomorphic to the
standard $S^6$.
\end{proof}

\section{The canonical bundle and deformations}\label{geom}
In this section, we discuss some complex analytic properties of the new complex structures. 
\subsection{The anticanonical pencil}

\begin{proposition}\label{can}
The threefold $X$ satisfies
\begin{align}\label{canon}
 K_X\simeq f^*\cO_{\PP^1}(-1).
\end{align}
Moreover, $\CC(X)=f^*\CC(\PP^1)$, and $f$ is the morphism defined by
the complete anticanonical pencil.
\end{proposition}

\begin{proof}
For the complement of the zero section in a line bundle,
$K_{M^\times}$ is the pullback of $K_S$. Locally it is represented by
a two-form on $S$ wedged with $dv/v$. Translation on the elliptic
curve preserves its invariant differential. Multiplication in the
line-bundle direction changes $dv/v$ by a one-form from $S$, which
vanishes when wedged with the base two-form. Thus the canonical
identification descends through the lifted translation. Since
$K_S=\pi^*\cO(-1)$, \eqref{canon} holds over $U$.

We check the extension at the exceptional fibers. In a Mumford
chart, a local canonical form on the punctured family is
\begin{align}
 dq\wedge\frac{dz_1}{z_1}\wedge\frac{dz_2}{z_2}.
\end{align}
The rays of the fan have height one and the cones are unimodular.
Substitution in a rank-two chart $q=x_0x_1x_2$ gives a nonzero
constant multiple of $dx_0\wedge dx_1\wedge dx_2$. In rank one it
gives the analogous nowhere-zero form in the nodal chart times the
elliptic differential. Hence no vertical canonical divisor is added
at the three Mumford fibers.

At infinity use the variables in \eqref{inf} and \eqref{good}.
If $v'$ is the line coordinate after the section-preserving
comparison, then
\begin{align}\label{vol}
 ds\wedge\frac{dX}{Y}\wedge\frac{dv}{v}
 =4\,dr\wedge\frac{dx'}{y'}\wedge\frac{dv'}{v'}.
\end{align}
The powers $r^3$ from $ds$ and $r^{-3}$ from the elliptic
differential cancel. The base factor $r^3$ in the line comparison
adds a multiple of $dr/r$, which vanishes in this wedge product.
The right hand side is nonvanishing, invariant under \eqref{deck},
and unchanged by the logarithmic translation. It descends through
the free quotient. Thus no additional multiple-fiber term occurs,
and \eqref{canon} holds on $X$.

By Theorem~\ref{sphere}, $H^2(X,\ZZ)=0$. If a divisor $D$ dominated
the base, its intersection with a general torus fiber would be a
nonzero effective divisor of zero first Chern class. Intersecting
with a K\"ahler class on that torus is a contradiction. All divisors
are therefore vertical. The polar divisor of a meromorphic function
is vertical, so the function restricts to a holomorphic, hence
constant, function on a general fiber. Properness and connectedness
of the fibers imply descent to the base. This proves
$\CC(X)=f^*\CC(\PP^1)$.

Finally, $f_*\cO_X=\cO_{\PP^1}$. The projection formula and
\eqref{canon} give
\begin{align}\label{anti}
 H^0(X,-K_X)=H^0(\PP^1,\cO(1))\simeq\CC^2.
\end{align}
These sections have no common zero and define $f$.
\end{proof}

\begin{corollary}\label{new}
The threefold $X$ is not biholomorphic to a member of the original
Alp\"oge--Claude family.
\end{corollary}

\begin{proof}
The fibration in \cite{AC} is its algebraic reduction. Hence a
biholomorphism with $X$ would identify the two fibrations up to an
automorphism of $\PP^1$. It would preserve the exceptional values
and the scheme-theoretic multiplicities. Our fibration has four
exceptional values and one multiple fiber of multiplicity four.
The original configuration has three exceptional values and two
multiple fibers of multiplicities three and four. They cannot be
identified.
\end{proof}

\begin{remark}
There is also an intrinsic distinction in the anticanonical systems
between our complex structure and that of Alp\"oge-Claude. 
The formula in \cite[Section~9.2]{AC} is
\begin{align}
 K_{X_{\mathrm{AC}}}
 \simeq f_{\mathrm{AC}}^*\cO_{\PP^1}(-1)\otimes\cO(2B_4),
 \qquad 4B_4=f_{\mathrm{AC}}^*(p_4).
\end{align}
A section of $-K_{X_{\mathrm{AC}}}$ is therefore a pullback of a
section of $\cO_{\PP^1}(1)$ which vanishes at $p_4$. Thus
$h^0(X_{\mathrm{AC}},-K)=1$, whereas \eqref{anti} gives two for $X$.
\end{remark}

\subsection{Hodge numbers}
We only make some brief remarks on the Hodge numbers, and do not include any detailed calculations. 
The coherent Leray calculation is similar to
\cite[Proposition~9.13]{AC}. For this construction it gives
\begin{align}\label{direct}
 f_*\cO_X=\cO_{\PP^1},\qquad
 R^1f_*\cO_X\simeq\cO_{\PP^1}\oplus\cO_{\PP^1}(-1),\qquad
 R^2f_*\cO_X\simeq\cO_{\PP^1}(-1).
\end{align}
Here we use the cohomology of the non-reduced fiber $4B_*$, rather
than that of its reduction $B_*$.
Each fiber $D$ has
$h^q(D,\cO_D)=(1,2,1)$ for $q=0,1,2$, so Grauert's theorem gives
local freeness and base change. In other words, $f$ is cohomologically
flat in all degrees. The map $f$ is proper and flat, and its fibers
are Cartier divisors in the smooth threefold $X$, hence Gorenstein.
Analytic relative duality of Ramis--Ruget--Verdier \cite{RRV} therefore
gives
\begin{align*}
 (R^2f_*\cO_X)^*\simeq f_*\omega_{X/\PP^1},\qquad
 \omega_{X/\PP^1}=K_X\otimes f^*K_{\PP^1}^{-1}
 \simeq f^*\cO_{\PP^1}(1).
\end{align*}
Since $f_*\cO_X=\cO_{\PP^1}$, the projection formula yields
$R^2f_*\cO_X\simeq\cO_{\PP^1}(-1)$.

The invariant real class $\psi$ gives a
nowhere-zero section of $R^1f_*\cO_X$: this follows from the
normalization sequences at the three reduced singular fibers and by
restriction to $B_*$ at the multiple fiber. Riemann--Roch gives
$\deg R^1f_*\cO_X=-1$, and the resulting extension of $\cO(-1)$ by
$\cO$ splits because $H^1(\PP^1,\cO(1))=0$.
The coherent Leray spectral sequence then yields
\begin{align}\label{hodge}
 h^{0,0}(X)=1,\qquad h^{0,1}(X)=1,\qquad
 h^{0,2}(X)=h^{0,3}(X)=0.
\end{align}
We do not compute the remaining Hodge numbers.

\subsection{A two-parameter family}
The key observation is that the elliptic input in our construction also varies; its parameter space is one-dimensional
by \cite[Section~13.9]{SS}. A convenient local family is
\begin{align}\label{mu}
 S_\mu:\quad y^2=x^3+tx+\mu t+1,\qquad
 P_\mu=\left(-\mu,\sqrt{1-\mu^3}\right),
\end{align}
where the square root is the holomorphic branch equal to one at
zero. The discriminant polynomial at finite $t$ is, up to its
nonzero constant factor,
\begin{align}\label{disc}
 d_\mu(t)=4t^3+27(\mu t+1)^2,
 \qquad\operatorname{disc}_t(d_\mu)=314928(\mu^3-1).
\end{align}
It has three distinct roots near $\mu=0$. At infinity, the equation
is $Y^2=X^3+s^3X+\mu s^5+s^6$, still of type $III^*$.
After $s=r^4$ and the same rescaling, the good-reduction equation is
\begin{align}\label{mugood}
 y'^2=x'^3+x'+\mu r^2+r^6,\qquad
 P'_\mu=(-\mu r^2,\sqrt{1-\mu^3}\,r^3).
\end{align}
In particular the limiting torsion point is unchanged.

\begin{proposition}\label{family}
The construction extends to the holomorphic family over the parameter
bidisc in Theorem~\ref{main}. Its image in the local deformation space
has dimension two.
\end{proposition}

\begin{proof}
Label the roots $p_i(\mu)$ holomorphically near zero and put the
linearization zero at $p_3(\mu)$. The minimal regular elliptic
surfaces form an equisingular family; after shrinking the parameter
disc, the resolutions and the sections can be chosen simultaneously.
The fiber components, the intersection numbers, and the height
$\langle P_\mu,P_\mu\rangle=1/2$ are unchanged. Thus
\begin{align}
 \sHom(t_{2P_\mu}^*\cO(P_\mu-O_\mu),\cO(P_\mu-O_\mu))
 \simeq\pi_\mu^*\cO_{\PP^1}(1).
\end{align}
The one-dimensional space of its sections vanishing at the chosen
root varies holomorphically. Choose a holomorphic generator and
multiply it by $\lambda$.

The toroidal charts and their lattice identifications are fixed,
while their unit factors vary holomorphically. To check the
linearization at infinity uniformly in $\mu$, put
$b=\sqrt{1-\mu^3}$ and use the tangent-line function
\begin{align}
 g_\mu=\frac{y-b-\frac{t+3\mu^2}{2b}(x+\mu)}{x+\mu}.
\end{align}
It has the same divisor formula as $g$ in Lemma~\ref{lin}.
In the coordinates of \eqref{mugood}, define
\begin{align}\label{gm}
 G_\mu=
 \frac{r^3(y'-br^3)}{x'+\mu r^2}
 -\frac{1+3\mu^2r^4}{2b}.
\end{align}
Then $g_\mu=r^{-4}G_\mu$ and
\begin{align}
 \frac{t-p_3(\mu)}{g_\mu}
 =\frac{1-p_3(\mu)r^4}{G_\mu},\qquad
 G_\mu|_{r=0}=-\frac{1}{2b}.
\end{align}
As a meromorphic function, $G_\mu$ has the horizontal divisor
$D'_\mu$ and no vertical component. Let $\eta_{D'_\mu}$ denote the
meromorphic section with divisor $D'_\mu$, as in the proof of
Lemma~\ref{lin}. Then
$\lambda(1-p_3(\mu)r^4)\eta_{D'_\mu}/G_\mu$ is an invertible
linearization on the whole good-reduction family. The function
$G_\mu$ is invariant under the same deck action as in \eqref{deck}.
The coordinates $z=-X/Y$ and $z'=-x'/y'$ still satisfy
$z=r^{-3}z'$, because the rescalings $X=r^6x'$ and $Y=r^9y'$ do
not depend on $\mu$. As explained after \eqref{con}, this is the
comparison of the local frames along the zero sections. Its remaining
holomorphic unit can be chosen holomorphically in $\mu$. Hence the
section-preserving comparison of Lemma~\ref{comp} is also simultaneous.

We now work on the parameter bidisc in Theorem~\ref{main}, centered
at $(0,\lambda_0)$, where $\lambda_0\ne0$. Since this bidisc is simply
connected, the integral markings can be chosen consistently, and
we keep the vector $\ell$ fixed in these markings. The resulting
four-torsion section $\ell/4$, the free affine action, and all four
gluing identifications therefore vary holomorphically in
$(\mu,\lambda)$. After shrinking this bidisc, compactness gives a
common contraction bound. The fixed Mumford fillings and the good-reduction 
quotient are smooth over the parameter space, so
the resulting family is a proper holomorphic submersion. Its fibers
are diffeomorphic by Ehresmann's theorem, and hence all have the
topology proved in Theorem~\ref{sphere}. The canonical calculation
is unchanged.

We show that the two parameters describe independent variations of
the complex threefold. The anticanonical pencil is intrinsic. Its
four exceptional values have one distinguished multiple value and
one distinguished hexagonal value. After fixing the multiple value
at infinity and locally labelling the two rank-one values, their
cross-ratio is therefore an invariant of the complex threefold.
Let $a_i=p_i(0)$. Differentiation of \eqref{disc} gives
\begin{align}\label{rootd}
 p_i'(0)=-\frac{9}{2a_i}.
\end{align}
An infinitesimally constant configuration modulo automorphisms of
$\PP^1$ fixing infinity would satisfy
$p_i'(0)=Aa_i+B$ for some constants $A,B$. By \eqref{rootd}, the
quadratic polynomial $Az^2+Bz+9/2$ would vanish at all three distinct
$a_i$, which is impossible. Thus the cross-ratio has nonzero
$\mu$-derivative.

For fixed $\mu$, consider the normalized double curve of one
rank-one fiber. In \eqref{periods} its elliptic quotient is
\begin{align}\label{eq}
 E_{\mu,\lambda}
 =\CC^*/\langle c(\mu)\lambda\rangle,
 \qquad c(\mu)\ne0.
\end{align}
Here $c(\mu)$ is the unit value of the unscaled multiplier at that
nodal point. The $j$-invariant varies nontrivially with $\lambda$:
for $q=c(\mu)\lambda$ sufficiently small,
\begin{align}
 j(q)=q^{-1}+744+O(q).
\end{align}
The cross-ratio is independent of $\lambda$, while this elliptic
invariant has nonzero $\lambda$-derivative. They therefore give two
independent local invariants. The possible interchange of the two
rank-one fibers is a finite ambiguity and does not change the
dimension. A locally trivial deformation of the total complex
manifold induces a trivial deformation of its intrinsic
anticanonical pencil and of these double curves. Consequently the
two variations cannot both be absorbed by biholomorphisms. Their
image in a local Kuranishi space has dimension two.
\end{proof}

Proposition~\ref{family}, Theorem~\ref{sphere}, Proposition~\ref{can},
and Corollary~\ref{new} complete the proof of Theorem~\ref{main}.
\bibliographystyle{amsplain}
\bibliography{S6_References}

\providecommand{\bysame}{\leavevmode\hbox to3em{\hrulefill}\thinspace}
\providecommand{\MR}{\relax\ifhmode\unskip\space\fi MR }
\providecommand{\MRhref}[2]{%
  \href{http://www.ams.org/mathscinet-getitem?mr=#1}{#2}
}
\providecommand{\href}[2]{#2}
\begin{thebibliography}{10}

\bibitem{CV}
Carlos~A. Cadavid and Juan~D. V{\'e}lez, \emph{Normal factorization in
  {$SL(2,\mathbb{Z})$} and the confluence of singular fibers in elliptic
  fibrations}, Beitr. Algebra Geom. \textbf{50} (2009), no.~2, 405--423.

\bibitem{Chen}
Zhangchi Chen, \emph{A conditional {Oka} complex structure of the six-sphere},
  \href{https://arxiv.org/abs/2609.26706}{arXiv:2609.26706}, 2026.

\bibitem{AC}
{Claude} and Levent Alp{\"o}ge, \emph{The $(3,4,\infty)$ modular family of
  $2$-tori, completed at its three special points, is a complex structure on
  {$S^6$}}, Preprint, \url{https://alpo.ge/s6.pdf}, 2026.

\bibitem{Engel}
Philip Engel, \emph{Complex structures on {$S^6$}}, Preprint,
  \url{https://philip-engel.github.io/S6.pdf}, 2026.

\bibitem{EGS}
Philip Engel, Olivier de~Gaay~Fortman, and Stefan Schreieder,
  \emph{Combinatorics and {Hodge} theory of degenerations of abelian varieties:
  {A} survey of the {Mumford} construction},
  \href{https://arxiv.org/abs/2507.15695}{arXiv:2507.15695v3}, 2026.

\bibitem{HV}
Nobuhiro Honda and Jeff Viaclovsky, \emph{Fibrations on the $6$-sphere and
  {Clemens} threefolds}, Adv. Math. \textbf{503} (2026), no.~part B, 111221.

\bibitem{KM}
Michel~A. Kervaire and John~W. Milnor, \emph{Groups of homotopy spheres. {I}},
  Ann. of Math. (2) \textbf{77} (1963), 504--537.

\bibitem{LR26}
Wenfei Liu and S{\"o}nke Rollenske, \emph{Yet another family of complex
  structures on {$S^6$}}, Preprint, 2026.

\bibitem{Mumford}
David Mumford, \emph{An analytic construction of degenerating abelian varieties
  over complete rings}, Compositio Math. \textbf{24} (1972), 239--272.

\bibitem{RRV}
Jean-Pierre Ramis, Gabriel Ruget, and Jean-Louis Verdier, \emph{Dualit{\'e}
  relative en g{\'e}om{\'e}trie analytique complexe}, Invent. Math. \textbf{13}
  (1971), no.~4, 261--283.

\bibitem{SS}
Matthias Sch{\"u}tt and Tetsuji Shioda, \emph{Elliptic surfaces}, Adv. Stud.
  Pure Math. \textbf{60} (2010), 449--545,
  \href{https://arxiv.org/abs/0907.0298}{arXiv:0907.0298v3}.

\bibitem{Shioda}
Tetsuji Shioda, \emph{On the {Mordell--Weil} lattices}, Comment. Math. Univ.
  St. Pauli \textbf{39} (1990), no.~2, 211--240.

\bibitem{Wang}
Zichang Wang, \emph{A complex structure on a rational homology six-sphere with
  two-torsion}, \href{https://arxiv.org/abs/2609.11541}{arXiv:2609.11541v1},
  2026.

\bibitem{Yang}
Xiufan Yang, \emph{Compact complex threefolds fibred by two-dimensional complex
  tori}, \href{https://arxiv.org/abs/2609.29666}{arXiv:2609.29666}, 2026.

\end{thebibliography}
\end{document}